\documentclass[reqno]{amsart}

\usepackage{a4,amsmath,amssymb,amscd,verbatim,bbm,graphicx, relsize}
\newtheorem{theorem}{Theorem}
\newtheorem{proposition}[theorem]{Proposition}
\newtheorem{corollary}[theorem]{Corollary}
\newtheorem{lemma}[theorem]{Lemma}
\theoremstyle{definition}

\newtheorem{definition}[theorem]{Definition}

\newtheorem{remark}[theorem]{Remark}

\newtheorem*{theorem*}{Theorem}
\newtheorem*{corollary*}{Corollary}
\newtheorem*{proposition*}{Proposition}
\newtheorem*{remark*}{Remark}
\numberwithin{equation}{section}
\numberwithin{theorem}{section}

\author{Stephen Cantrell}
\address{Mathematics Institute, University of Warwick,
Coventry CV4 7AL, U.K.}
\email{S.J.Cantrell@warwick.ac.uk}

\title{Typical behaviour along geodesic rays in hyperbolic groups}

\begin{document}
\maketitle
\begin{abstract}
In this note we study the limiting behaviour of real valued functions on hyperbolic groups as we travel along typical geodesic rays in the Gromov boundary of the group. Our results apply to group homomorphisms, certain quasimorphisms and to the displacement functions associated to convex cocompact group actions on CAT$(-1)$ metric spaces.
\end{abstract}
\section{Introduction}
Let $G$ by a non-elementary hyperbolic group and suppose that $G$ acts cocompactly (or convex cocompactly) by isometries on a complete hyperbolic geodesic metric space $(X,d)$. Fix a finite generating set $S$ for $G$ and an origin $o$ for $X$. Let $C(G)$ denote the Cayley graph of $G$ with respect to $S$ and write $\partial G$ for the Gromov boundary of $G$. By the \u Svarc-Milnor Lemma, there exists constants $C_1,C_2>0$ such that,
for any infinite geodesic ray $\gamma$ based at the identity in $C(G)$,
$$ C_1 n\le d(o,\gamma_no) \le C_2n$$
for all $n \ge 1$. Here $\gamma_n$ denotes the end point of $\gamma$ after $n$ steps. This inequality describes the coarse behaviour of the displacement function $g\mapsto d(o,go)$ along geodesic rays. It is then natural to ask whether we can describe more precisely how the displacement grows along typical geodesic rays in $\partial G$? The Patterson-Sullivan measure provides us with a natural way of quantifying typicality in this setting. We say that a property exhibited by elements of $\partial G$ is typical if it holds on a full Patterson-Sullivan measure set. \\ \indent
Gekhtman, Taylor and Tiozzo  asked the above question in a more general setting. They prove the following theorem in \cite{lox}. Let $\nu$ denote the Patterson-Sullivan measure obtained as the weak star limit
$$\lim_{n\to\infty} \frac{\sum_{|g|\le n} \lambda^{-|g|}\delta_g}{\sum_{|g|\le n} \lambda^{-|g|}},$$
where $\delta_g$ denotes the Dirac measure based at $g \in G$ and $|g|$ denotes the word length of $g$. We write $[\gamma] \in \partial G$ for the element in $\partial G$ that contains $\gamma$.
\begin{proposition} [Theorem $1.3$ \cite{lox}]
Suppose a hyperbolic group $G$ has a non-elementary action by isometries on a separable, hyperbolic geodesic metric space $X$. Then, there is $L>0$ such that for every $x \in X$ and $\nu$ almost every $[\widetilde{\gamma}] \in \partial G$,
$$\lim_{n\to\infty} \frac{d_X(x,\gamma_n x)}{n} = L,$$
where $\gamma$ is any geodesic ray in $[\widetilde{\gamma}]$. 
\end{proposition}
\noindent To prove this, Gekhtman, Taylor and Tiozzo exploit the strongly Markov structure of $G$. That is, they use the fact that there exists a finite directed graph $\mathcal{G}$ that in some sense encodes the key properties of $G$. They obtain the above theorem by studying random walks on the loop graph associated to $\mathcal{G}$.\\ \indent This is one way to exploit the structure provided by $\mathcal{G}$. It is however possible to make use of the strongly Markov property in a different way. The graph $\mathcal{G}$ gives rise to a dynamical system $(\Sigma, \sigma:\Sigma \to \Sigma)$ known as a subshift of finite type. We can embed $G$ into $\Sigma$ via a function $i:G \to \Sigma$ and use this to translate questions about the displacement function on $G$ to questions about $\Sigma$ and a suitable function $f: \Sigma \to \mathbb{R}$. The connection between $G$ and $\Sigma$ is exploited by Pollicott and Sharp in \cite{45}. They prove an almost sure invariance principle, as well as other limit laws, for the displacement function associated to the action of surface groups and convex cocompact free groups on the hyperbolic plane. In \cite{me} similar ideas are used to derive limit laws for real-valued functions satisfying two conditions named, in that paper, by Condition $(1)$ and Condition $(2)$. Real valued group homomorphisms, certain quasimorphisms as well as the displacement function associated to convex cocompact group actions on $\text{CAT}(-1)$ metric spaces satisfy these conditions.\\ \indent
This leads us to ask whether Proposition $1.1$ remains true if we replace the displacement function with a different real valued function. Furthermore, can we formulate a more precise statement describing how these functions behave along geodesic rays? These are the questions that we consider in this paper. Our main theorems are the following. We will define and discuss Condition $(1)$ and Condition $(2)$ in Section $3$. Let $\nu$ denote the Patterson-Sullivan measure as defined above.
\begin{theorem}\label{thm1}
Let $G$ be a non-elementary hyperbolic group equipped with a finite generating set $S$. Suppose that $\varphi:G\to\mathbb{R}$ satisfies Condition $(1)$ and Condition $(2)$. Then there exists $\Lambda \in \mathbb{R}$ such that for $\nu$ almost every $[\widetilde{\gamma}]  \in \partial G$,
$$ \lim_{n\to\infty} \frac{\varphi(\gamma_n)}{n} = \Lambda,$$
for any $\gamma$ belonging to $[\widetilde{\gamma}]$.
\end{theorem}
\begin{remark}
When $\varphi$ is the displacement function associated to a convex cocompact group action on a $\text{CAT}(-1)$ metric space, we recover a special case of Proposition $1.1$. We note that the non-elementary actions to which Proposition $1.1$ applies are more general than convex cocompact.
\end{remark}
This shows that, along typical elements of $\partial G$, a function $\varphi$ satisfying the hypotheses of Theorem $1.2$ grows asymptotically like $\Lambda n$.  We can then ask if it is possible to describe more precisely how $\varphi$ grows along elements of $\partial G$. To achieve this, we need to impose an additional assumption on $\varphi$ to ensure that $\varphi(\cdot)-|\cdot|\Lambda$ grows along typical geodesic rays. Specifically, we need that the set
$$\left\{[\gamma] \in \partial G: \{\varphi(\gamma_n)-n\Lambda:n\in\mathbb{Z}_{\ge0}\} \text{ is unbounded}\right\}$$
is non-empty. The fact that this set is well-defined will follow from Condition (2).  Surprisingly, this is the only additional hypothesis we need in order to obtain the following, more precise description of how $\varphi$ grows.
\begin{theorem}
Let $G$ be a non-elementary hyperbolic group equipped with a finite generating set $S$. Fix a bounded subset $H$ of the Cayley graph of $G$. Suppose $\varphi:G\to\mathbb{R}$ satisfies Condition $(1)$ and Condition $(2)$ and that $\Lambda$ is the quantity defined in Theorem $1.2$. Then, if the set
$$\left\{[\gamma] \in \partial G: \{\varphi(\gamma_n)-n\Lambda:n\in\mathbb{Z}_{\ge0}\} \text{ is unbounded}\right\}$$
is non-empty, there exists $\sigma^2>0$ such that for $x\in\mathbb{R}$,
$$ \nu(\mathcal{A}_n(x)) =  \frac{1}{\sqrt{2 \pi } \sigma} \int_{-\infty}^x e^{-t^2/2\sigma^2} \ dt + O(n^{-1/4}) ,$$
as $n \to \infty$, where
$$ \mathcal{A}_n(x)=\left\{ [\widetilde{\gamma}] \in \partial G : \text{ for all } \gamma \in [\widetilde{\gamma}] \text{ with } \gamma_0 \in H, \frac{\varphi(\gamma_n)-n\Lambda}{\sqrt{n}} \le x \right\}.$$
The implied constant is uniform in $x \in \mathbb{R}$.
\end{theorem}
\begin{remark}
The reason that we ask for $\gamma_0 \in H$ is due to the following fact. For $\nu$ almost every $[\widetilde{\gamma}] \in \partial G$ and every $n \ge 1$, we can find $\gamma \in [\widetilde{\gamma}]$ for which $\varphi(\gamma_n) - n\Lambda$ is arbitrarily large. Therefore without this assumption, $\mathcal{A}_n$ would have zero $\nu$ measure for all $n\in\mathbb{Z}_{\ge 0}$.
\end{remark}
 The following result from \cite{me} then shows that real-valued group homomorphisms satisfy the hypotheses of Theorem $1.4$.
\begin{proposition} [\cite{me} Lemma 7.11, Corollary 7.12]
Let $G$ be a non-elementary hyperbolic group equipped with a finite generating set $S$. Suppose $\varphi:G\to\mathbb{R}$ is a non-trivial group homomorphism. Then the constant $\Lambda$ obtained from Theorem \ref{thm1}  is zero  and the set
$$\{[\gamma] \in \partial G: \{\varphi(\gamma_n):n\in\mathbb{Z}_{\ge0}\} \text{ is unbounded}\},$$
is non-empty and in fact has full $\nu$ measure.
\end{proposition}

To conclude the introduction, we briefly outline the contents of this paper. In the second section we cover preliminary material concerning hyperbolic groups, their strongly Markov structure and the Patterson-Sullivan measure. In the third section we discuss the regularity conditions, Condition $(1)$ and Condition $(2)$. We then, in Section $4$, study the properties of the Patterson-Sullivan measure. We prove Theorems $1.2$ and $1.4$ in the remaining section.\\

\textit{Notation:} Throughout the paper, we use the following notation to describe the asymptotic behaviour of sequences. Suppose $f_n, g_n,h_n$ are real valued sequences. We write $f_n = O(g_n)$ if there exists $C>0$ such that eventually $|f_n| \le C|g_n|$. If $|f_n/g_n| \to 0$ as $n\to\infty$ we write $f_n = o(g_n)$. We write $f_n = O(g_n,h_n)$ if $f_n=O(\max\{|g_n|,|h_n|\})$.

\section{Hyperbolic Groups and Symbolic Codings}
In this section we cover preliminary material related to hyperbolic groups and symbolic dynamics.
\begin{definition}
Let $G$ be a finitely generated group with finite generating set $S$. We define the left and right word metrics on $G$ by
$$d_L(g,h) = |g^{-1}h| \hspace{1mm} \text{ and } \hspace{1mm} d_R(g,h)=|gh^{-1}|$$
for $g,h\in G$. Here $|\cdot|$ denotes the word metric, i.e. $|g|$ is the length of the shortest word(s) representing $g$ with letters in $S\cup S^{-1}$. We say that $G$ is hyperbolic if there exists $\delta\ge0$ such that any geodesic triangle in the $d_L$ metric is $\delta$-thin (i.e. any point on the side of a geodesic triangle is within distance $\delta$ of one of the other two sides). 
\end{definition} 
We say that a hyperbolic group is non-elementary if it is not virtually cyclic, i.e. it does not contain a finite index cyclic subgroup. Suppose that $G$ is a non-elementary hyperbolic group equipped with a finite  generating set and let $W(n) = \# \{g \in G: |g|=n\}$ denote the word length counting function. Coornaert proved that the growth rate of $W(n)$ is purely exponential \cite{coor}, i.e. there exists $\lambda>1$ and $C_0,C_1 >0$ such that
$$C_0 \lambda^n \le W(n) \le C_1 \lambda^n.$$
This fact will be key to our analysis. \\ \indent
Let $C(G)$ denote the Cayley graph of $G$ with respect to $S$. The Gromov boundary $\partial G$ of $G$ consists of equivalence classes of infinite geodesic rays in $C(G)$. Two geodesic rays $\gamma$ and $\gamma'$ are said to be equivalent if $d_L(\gamma_n,\gamma'_n)$ is bounded uniformly for $n\in\mathbb{Z}_{\ge0}$. Here, $\gamma_n, \gamma'_n$ denote the end points of $\gamma$, $\gamma'$ after $n$ steps. Given an infinite geodesic ray $\gamma$ we use $[\gamma]$ to denote the element of $\partial G$ containing $\gamma$. There is a natural compact topology for $G \cup \partial G$ that extends the topology on $G$ given by the word metric. The action of $G$ extends continuously to $G \cup \partial G$ by sending $[\gamma] \in \partial G$ to $[g\gamma] \in \partial G$. \\
\indent The Patterson-Sullivan measure $\nu$ is a measure on $\partial G$ obtained as the weak star limit, as $n \to \infty$, of the following sequence of measures
$$\frac{\sum_{|g|\le n} \lambda^{-|g|}\delta_g}{\sum_{|g|\le n} \lambda^{-|g|}}$$
on $G \cup \partial G$. Here $\delta_g$ denotes the Dirac measure based at $g \in G$. The measure $\nu$  is ergodic with respect to the action of $G$ on $\partial G$. See \cite{coor} and \cite{patcat} for a comprehensive account of the above material concerning the Patterson-Sullivan measure. We will now discuss the combinatorial properties of hyperbolic groups.\\ \indent
As mentioned in the introduction, hyperbolic groups have nice combinatorial properties that arise due to their strongly Markov structure.
\begin{definition}
A finitely generated group $G$ is strongly Markov if given any finite generating set $S$ there exists a finite directed graph $\mathcal{G}$ with vertex set $V$, edge set $E$ and a labeling map $\rho : E \to S$ such that:
\begin{enumerate}
\item there exists an initial vertex $\ast \in V$ such that no directed edge ends at $\ast$;
\item the map taking finite paths in $\mathcal{G}$ starting at $\ast$ to $G$ that sends a path with concurrent edges $(\ast,x_1),...,(x_{n-1},x_n)$ to $\rho(\ast,x_1)\rho(x_1,x_2)...\rho(x_{n-1},x_n)$, is a bijection; 
\item the word length of $\rho(\ast,x_1)...\rho(x_{n-1},x_n)$ is $n$.
\end{enumerate}
\end{definition}
Cannon introduced this property and proved that cocompact Kleinian groups are strongly Markov \cite{can}. Ghys and de la Harpe showed that Cannon's method worked for arbitrary hyperbolic groups.

\begin{proposition} [\cite{gh} Theorem 13]
Any hyperbolic group is strongly Markov.
\end{proposition}
Throughout the rest of this paper we will assume that $G$ is a non-elementary hyperbolic group equipped with a finite  generating set $S$. Let $\mathcal{G}$ be a graph associated to a $G$ via the strongly Markov property. We augment $\mathcal{G}$ by adding an extra vertex $0\in V$ and edges $(v,0)$ for all $v\in V\cup\{0\}\backslash \{\ast\}$. We define $\rho(v,0) = e$ for $v\in V\cup\{0\}\backslash \{\ast\}$ , where $e \in G$ is the identity element. We will assume that any graph $\mathcal{G}$ associated to $G$ has been augmented in this way.\\ \indent
As mentioned in the introduction, we can use this strongly Markov structure to construct a dynamical system that encodes the properties of $G$. Suppose that $\mathcal{G} = (E,V)$ is a directed graph associated to $G$ via the strongly Markov property. We define a transition matrix $A$, indexed by $V \times V$,  by
\[
  A(v_1,v_2) = \left\{
     \begin{array}{@{}l@{\thinspace}l}
        1 & \ \  \text {if} \hspace{2mm}  (v_1,v_2) \in E \\
        0 & \ \  \text{otherwise.}
     \end{array}
   \right.
\]
Using $A$ we define
$$\Sigma_A = \{ (x_n)_{n=0}^\infty : x_n \in V \text{ and } A(x_n,x_{n+1}) = 1 \text{ for all } n\in\mathbb{Z}_{\ge 0}\}$$
and $\sigma:\Sigma_A \to \Sigma_A$ by $\sigma((x_n)_{n=0}^\infty) = (x_{n+1})_{n=0}^\infty$. The system $(\Sigma_A,\sigma)$ is known as a subshift of finite type. We embed $G$ into $\Sigma_A$ via the function $i:G\to\Sigma_A$ that sends a group element $g \in G$ to the unique element $(\ast,x_1,x_2,...,x_n,0,0,...)$ for which $\rho(\ast,x_1)...\rho(x_{n-1},x_n) = g$ and $|g|=n$. This correspondence will allow us to prove facts about $G$ by studying the properties of $\Sigma_A$. For the rest of this section we recount the properties of subshifts that we require for our proofs.\\ \indent
Let $B$ be a zero-one matrix. We say that $B$ is irreducible if given $i,j$, there exists $N$ such that $B^N(i,j) >0$. If there exists $N$ such that $B^N(i,j)>0$ for all pairs $i,j$ then we say that $B$ is aperiodic. For each $0<\theta <1$ there is a metric $d_\theta$ on $\Sigma_B$ defined by $d_\theta(x,y) = \theta^{s(x,y)}$ where $s(x,y) \in \mathbb{Z}_{\ge0}$ is the first integer $n$ such that $x_n \neq y_n$. We write $F_\theta(\Sigma_B) = \{ f: \Sigma_B \to \mathbb{R} : \text{ $f$ is Lipschitz in the $d_\theta$ metric}\}.$ Given $f \in F_\theta(\Sigma_B)$, we write $f^n(x)=f(x)+f(\sigma(x))+...+f(\sigma^{n-1}(x))$ for $x \in \Sigma_B$. Throughout the following, we assume that $B$ is irreducible. When this is the case, the system $(\Sigma_B,\sigma)$ is transitive and admits a unique measure of maximal entropy $\mu$ \cite{parry}, i.e. there exists unique $\mu$ such that
$$\sup_{\nu} h_{\nu}(\sigma) = h_{\mu}(\sigma),$$
where the above supremum is taken over all $\sigma$-invariant probability measures. The measure $\mu$ is ergodic with respect to $\sigma$. If $f \in F_\theta(\Sigma_B)$  for some $0<\theta<1$ and $\int f \ d\mu=0$, then there exists $\sigma_f^2 \ge 0$ such that for $x \in \mathbb{R}$
$$\mu\left\{z \in \Sigma_B : \frac{f^n(z)}{\sqrt{n}}\le x \right\} = \frac{1}{\sqrt{2\pi} \sigma_f} \int_{-\infty}^x e^{-t^2/2\sigma_f^2} \ dt +O(n^{-1/2})$$
as $n\to\infty$ \cite{cp}. Furthermore, $\sigma_f^2 =0$ if and only if there exist continuous $h:\Sigma_B \to \mathbb{C}$ such that $f = h \circ \sigma - h$. In \cite{cp} this result is proved under the assumption that $B$ is aperiodic, however it is easy to see that this result passes to the irreducible case.\\ \indent
We note that since $\mathcal{G}$ has no edges that enter $\ast$, the matrix $A$ associated to $\mathcal{G}$ will never be irreducible. It is possible however that if we remove, from $A$, the rows/columns corresponding to the $0$ and $\ast$ vertices, then the resulting matrix is irreducible (or aperiodic). We say that $A$ is irreducible (or aperiodic) if this is the case. Although in general it is possible that $A$ is not irreducible, we can, by relabeling the vertex set $V$, assume $A$ has the form
$$A = \begin{pmatrix} 
A_{1,1} & 0 & \dots & 0  \\
A_{2,1} & A_{2,2} & \dots & 0\\
\vdots & \vdots & \ddots & \vdots\\
A_{m,1} & A_{m,2} & \dots & A_{m,m}
\end{pmatrix},$$
where $A_{i,i}$ are irreducible for $i=1,...,m$. We call the $A_{i,i}$ the irreducible components of $A$. Let $\lambda>1$ denote the exponential growth rate of $W(n)$. It is easy to see by Property $(2)$ and $(3)$ in Definition $2.2$ that all of the $A_{i,i}$ must have spectral radius at most $\lambda$. Furthermore there must be at least one $A_{i,i}$ with spectral radius exactly $\lambda$. We call an irreducible component maximal if it has spectral radius $\lambda$. We label the maximal components $B_i$ for $i=1,...,m$. The following key result follows from Coornaert's estimates on $W(n)$.
\begin{proposition} [\cite{cf} Lemma $4.10$]
The maximal components of $A$ are disjoint. There does not exist a path in $\mathcal{G}$ that begins in one maximal component and ends in another.
\end{proposition}
\section{Regularity Conditions}
In this section we discuss Condition $(1)$ and Condition $(2)$. This will be a brief survey of the functions satisfying these conditions, see Section $4$ of \cite{me} for a more comprehensive account. Condition $(1)$ and Condition $(2)$ are defined as follows. \medskip \\
 \noindent
\textit{Condition $(1)$} There exists a graph $\mathcal{G}$ associated to $G,S$ via the strongly Markov property with transition matrix $A$ and a function $f\in F_\theta(\Sigma_A)$ (for some $0<\theta<1$) such that $\varphi(g) = f^{|g|}(x)$ for $g \in G \text{ and } x = i(g) \in \Sigma_A.$\medskip \\
 \noindent
\textit{Condition $(2)$} $\varphi$ is Lipschitz in the left and right word metrics on $G$. \medskip

Although Condition $(1)$ relies on the properties of $\Sigma_A$, there is a natural assumption we can place on $\varphi:G \to \mathbb{R}$ to guarantee the existence of appropriate $\Sigma_A$ and $f: \Sigma_A \to \mathbb{R}$. Given $g,h \in G$, let $(g,h)$ denote their Gromov product
$$(g,h)=\frac{1}{2}\left(|g|+|h| - |gh^{-1}|\right).$$ 
\begin{definition}
We say that $\varphi:G\to\mathbb{R}$ is H\" older if for any fixed finite generating set $S$ and $a \in G$, there exists $C  >0$ and $0<\theta <1$ such that
$$|\Delta_a\varphi(g)-\Delta_a\varphi(h)| \le C \theta^{ (g,h)},$$
for any $g,h \in G$. Here, $\Delta_a\varphi (g)= \varphi(ag)-\varphi(g)$ for $a,g \in G$.
\end{definition}
Pollicott and Sharp prove that H\" older functions satisfy Condition $(1)$ in \cite{oc}. In \cite{cf} and \cite{me}, combable and edge combable functions are defined. We refer the reader to these papers for the definitions. Both these classes of functions satisfy Condition $(1)$, see Lemma $4.5$ in \cite{me}. It is clear that homomorphism to $\mathbb{R}$ are edge combable and so satisfy Condition $(1)$. The homomorphism property implies that real valued homomorphism also satisfy Condition $(2)$. In fact, the more general class of quasimorphism satisfy Condition $(2)$.
\begin{definition}
A function $\varphi : G \to \mathbb{R}$ is a quasimorphism if there exists a constant $A >0$ such that
$$|\varphi(gh) - \varphi(g) -\varphi(h)|\le A$$ for all $g,h \in G$.
\end{definition}
It is easy to check that quasimorphisms satisfy Condition $(2)$. In \cite{cf}, Calegari and Fujiwara show that Brooks counting quasimorphisms (see \cite{brooks} for a definition) satisfy Condition $(1)$ and so by the above discussion, our theorems apply to these functions. The following example, due to Barge and Ghys \cite{bg}, is a quasimorphism that satisfies the H\" older condition.\\ \indent
\textit{Example}: Suppose $G$ acts cocompactly by isometries on a simply connected Riemannian manifold $X$ with all sectional curvatures bounded above by $-1$. Write $M=X/G$. Given a smooth $1$-form $\omega$ on $M$, we can lift $\omega$ to a $G$-invariant smooth $1$-form $\widetilde{\omega}$ on $X$. Fix an origin $o \in X$ and define $\varphi: G \to\mathbb{R}$ by
$$\varphi(g) = \int_o^{go} \widetilde{\omega}.$$
Note that
$$\varphi(gh)-\varphi(g)-\varphi(h) = \int_{\partial T(g,h)} \widetilde{\omega} = \int_{T(g,h)} d\widetilde{\omega}$$
where $T(g,h)$ denotes the triangle in $\mathbb{H}$ with vertices $o,go$ and $gho$. By compactness and hyperbolicity, the right hand side of the above is bounded uniformly in $g,h$. This proves that $\varphi$ is a quasimorphism. In \cite{pic} Picaud proved that these quasimorphisms satisfy Condition $(1)$.\\ \indent
Another example of a function satisfying Condition $(1)$ and Condition $(2)$ was mentioned in the introduction. Suppose $G$ acts properly discontinuously, convex cocompactly by isometries on a complete $\text{CAT}(-1)$ geodesic metric space $(X,d)$. Fix a finite generating set for $G$ and an origin $o$ for $X$. A result of Pollicott and Sharp (Proposition $3$ from \cite{18}) proves that the displacement function satisfies Condition $(1)$. Furthermore, it is easy to see that this function satisfies Condition $(2)$. See Lemma $4.6$ of \cite{me} for a more detailed discussion.\\ \indent
This concludes our brief survey of functions satisfying Condition $(1)$ and Condition $(2)$. See \cite{bg}, \cite{ef} and \cite{gh} for further examples as well as Chapter $3$ of \cite{horshamthesis} for a more comprehensive account of these functions.

\section{Properties of the Patterson--Sullivan Measure}

The results presented in \cite{me} and \cite{lox} as well as this paper rely on the work of Calegari and Fujiwara \cite{cf} that compares the Patterson-Sullivan measure $\nu$ to a natural measure $\mu$ on $\Sigma_A$. In this section we construct this measure and compare it to $\nu$. To deduce our results we need to extend the work in \cite{cf} to obtain a deeper understanding of how the measures $\mu$ and $\nu$ compare. \\ \indent
Suppose $G$ has associated subshift $\Sigma_A$ which is obtained from the directed graph $\mathcal{G}$. Let $V$ denote the vertex set of $\mathcal{G}$. For $v \in \mathbb{R}^V$, define the function $p:\mathbb{R}^V \to \mathbb{R}^V$ by
$$p(v) = \lim_{n\to\infty} \frac{1}{n}\sum_{k=0}^n \frac{A^kv}{\lambda^k}.$$
This function projects $v$ to the eigenspace of $A$ corresponding to the eigenvalue $\lambda$. Similarly, the function $r: \mathbb{R}^V \to \mathbb{R}^V$ defined by
$$ r(v) = \lim_{n\to\infty} \frac{1}{n}\sum_{k=0}^n \frac{(A^T)^kv}{\lambda^k}$$
projects $v$ to the eigenspace of $A^T$ corresponding to the eigenvalue $\lambda$. To obtain the error term in Theorem $1.4$ we need to know the rate of convergence associated to the limit defining $p$.
\begin{lemma}
For $v \in \mathbb{R}^V$ we have that
$$p(v) =  \frac{1}{n}\sum_{k=0}^n \frac{A^kv}{\lambda^k} + O\left(\frac{1}{n}\right)$$
where the implied constant depends only on $v$.
\end{lemma}
\begin{proof}
Given $v \in \mathbb{R}^V$ we can write $v$ as a linear combination of elements in a Jordan basis for $A$. Since maximal components are disjoint, if an eigenvalue $x$ of $A$ has absolute value $\lambda$, then there does not exist a Jordan chain of length strictly greater than one associated to $x$. A simple calculation then shows that if $\widetilde{v}$ belongs to the generalised eigenspace associated to the eigenvalue $x\neq \lambda$, then
$$p(\widetilde{v}) = O\left(\frac{1}{n}\right).$$
The result follows.
\end{proof} 
Let $\textbf{1} \in \mathbb{R}^V$ denote the vector consisting of $1$ in each coordinate and let $v_\ast$ denote the vector consisting of a $1$ in the coordinate corresponding to the $\ast$ vertex and zeros elsewhere. Using $p$ and $r$, we define a measure $\mu$ on $\Sigma_A$ via a stochastic matrix $N : \mathbb{R}^V \to \mathbb{R}^V$ and vertex distribution $\rho: V\to \mathbb{R}$. For a vector $v\in \mathbb{R}^V$, let $v_j$ denote the coordinate of $v$ corresponding to the vertex $j \in V$. The matrix $N$ is defined as follows. If $p(\textbf{1})_i \neq 0$ then set
$$N_{i,j} = \frac{A_{i,j} p(\textbf{1})_j}{\lambda p(\textbf{1})_i}$$
and if $p(\textbf{1})_i=0$ let $N_{i,i}=1$ or $N_{i,j} = 0$ when $i\neq j$. 
The vertex distribution $\rho$ is defined by 
$$\rho(j) = p(\textbf{1})_j r(v_\ast)_j.$$
As for the usual construction of Markov measures, this defines a $\sigma$-invariant measure on $\Sigma_A$. We normalise this measure to obtain the probability measure $\mu$. There is a nice description of $\mu$ in terms of thermodynamic formalism.
\begin{proposition}
There exists $0< \alpha_i <1$ for $i=1,...,m$ with $\sum_{i=1}^m \alpha_i =1$ such that
\begin{equation}
\mu = \sum_{i=1}^m \alpha_i \mu_i,
\end{equation}
where each $\mu_i$ is the measure of maximal entropy for the system $(\Sigma_{B_i},\sigma)$.
\end{proposition}
\begin{proof}
Choose a maximal component $B_i$. One can check that the vector obtained from restricting $p(\textbf{1})$ or $r(v_\ast)$ to the vertices in $B_i$  is a right or left eigenvector respectively for $B_i$ (with eigenvalue $\lambda$). Then by comparing the construction of $\mu$ to Parry's construction of the measure of maximal entropy for a subshift of finite type \cite{parry}, we see that the restriction of $\mu$ to the maximal component $\Sigma_{B_i}$ is up to scaling, the measure of maximal entropy $\mu_i$ on this component. Furthermore, from the definitions of $p$ and $r$ and the fact that $\mu$ is $\sigma$-invariant, it is clear that $\mu$ assigns zero mass to the complement of the union of the maximal components. The result follows.
\end{proof}
Let $A'$ denote the matrix $A$ with the row/column corresponding to the $0$ vertex removed. 
\begin{definition}
Define sets $Y, Y_1,...,Y_m \subset \Sigma_{A'}$ by 
$$ Y = \{ x \in \Sigma_{A'} : x_0 =\ast\},$$
$$Y_i = \{x \in Y : x \text{ eventual enters } B_i \text{ and never leaves}\}.$$
Let $h: Y \to \partial G$ be the natural map associated to the bijection defined in Definition $2.2$. Given $y \in Y$, we use $h(y)_n$ to denote the $n$th step in the geodesic ray determined by $y$.
\end{definition}
There is a unique measure $\widehat{\nu}$ on $Y$ that pushes forward under $h$ to the Patterson-Sullivan measure on $\partial G$. We denote the pushforward map by $h_\ast$ so that $h_\ast\widehat{\nu} = \nu$. The measure $\widehat{\nu}$ can be constructed as in Section $4$ of \cite{cf}. We will not provide the construction here but will instead present the properties of $\widehat{\nu}$ that we require for our proofs. One of these properties is the following. We can explicitly calculate the $\widehat{\nu}$ measure of certain subsets of $\Sigma_{A'}$ called cylinder sets. Given a finite path in $\mathcal{G}$ let $[y]$ to denote the elements in $\Sigma_{A'}$ that have $y$ as an initial segment.
\begin{lemma}
Let $y$ be a finite path in $\mathcal{G}$ starting at $\ast$. We have that
$$\widehat{\nu}([y]) = \frac{p(\textup{\textbf{1}})_{v_y}}{p(\textup{\textbf{1}})_{\ast}} \lambda^{-|y|},$$
where $|y|$ is the length of $y$ and $v_y$ denotes the last vertex in $y$.
\end{lemma} 

\begin{proof}
This is a simple calculation that can be found in Section $4$ of \cite{cf}. Note that in this work, we are using a slightly different scaling for $\widehat{\nu}$. This introduces the $p(\textbf{1})_{\ast}$ term, which is not present in \cite{cf}.
\end{proof}
For $k \in \mathbb{Z}_{\ge 0}$, let $\sigma_\ast^k\widehat{\nu}$ denote the pushforward of $\widehat{\nu}$ under $\sigma^k$. The following lemma compares these pushforward measures to the measure $\mu$. 

\begin{lemma}
For each $v \in V$ with $\mu[v]>0$ and $k \in \mathbb{Z}_{\ge 0}$ there exists $\alpha_v^k \ge0$ such that
$$\sigma_\ast^k\widehat{\nu}|_{[v]} = \alpha_v^k \mu|_{[v]}.$$
There exists a length $k$ path from $\ast$ to $v$ if and only if $\alpha_v^k >0$. If $\mu[v]=0$ we define $\alpha_v^k = \widehat{\nu}(\sigma^{-k}[v])$ for all $k\in\mathbb{Z}_{\ge 0}$. Furthermore,  
\[   
 \frac{1}{n} \sum_{k=0}^n \alpha_v^k = 
     \begin{cases}
       1 +O(n^{-1}) &\text{ if } \mu[v]>0 \\
       O(n^{-1}) &\text{ if } \mu[v]=0.
     \end{cases}
\]
The implied constants can be taken to be independent of $v$ and $n$.
\end{lemma}
\begin{proof}
This is a consequence of Lemma $4.1$, the construction of $\widehat{\nu}$ and the proof of Lemma $4.22$ in \cite{cf}. A simple calculation using the definition of $\widehat{\nu}$ shows the existence of $\alpha_v^k$ satisfying the first condition of the lemma. The convergence associated to the final statement is proved in Lemma $4.22$ of \cite{cf}. By inspecting the proof of this lemma, we see that Lemma $4.1$ quantifies the convergence as $O(n^{-1})$.
\end{proof}
\noindent It follows that
$$\frac{1}{n} \sum_{k=0}^n \sigma_\ast^k \widehat{\nu}$$
converges in the weak star topology to the measure $\mu$. There is a much stronger relationship between $\widehat{\nu}$ and $\mu$ however. Given two measures, $\lambda_1$ and $\lambda_2$ on $\Sigma_{A}$, recall that their total variation $\|\lambda_1 - \lambda_2\|_{TV}$ is given by $\sup_{E \subset \Sigma_{A}} |\lambda_1(E)-\lambda_2(E)|$.

\begin{proposition}
We have that,
$$ \left\| \frac{1}{n} \sum_{j=0}^{n} \sigma_\ast^j \widehat{\nu}-\mu\right\|_{TV} = O(n^{-1})$$
as $n\to\infty$.
\end{proposition}

\begin{proof}
For any $E \subset \Sigma_{A}$,
\begin{align*}
\left| \frac{1}{n} \sum_{j=0}^{n} \sigma_\ast^j \widehat{\nu}(E)-\mu(E)\right| &= \left| \frac{1}{n} \sum_{j=0}^{n} \sum_{v\in V} \left( \sigma_\ast^j \widehat{\nu}|_{[v]}(E)-\mu|_{[v]}(E)\right) \right| \\
&\le \sum_{\substack{v\in V \\ \mu [v]>0}} \left|\frac{1}{n} \sum_{j=0}^{n} \alpha_v^j -1\right| + \sum_{\substack{v\in V \\ \mu [v]=0}} \left|\frac{1}{n} \sum_{j=0}^{n} \alpha_v^j\right|,
\end{align*}
where $\alpha_v^j$ are as defined in the previous lemma. Applying the previous lemma concludes the proof. 
\end{proof}
We will need the following definition and lemma later.
\begin{definition}
For each $j \in \mathbb{Z}_{\ge 0}$ let
$$A_j=\left(\sigma^{-j}\left(\bigcup_i \Sigma_{B_i}\right) \backslash \bigcup_{k=0}^{j-1} \sigma^{-k}\left(\bigcup_i \Sigma_{B_i}\right)\right) \cap Y.$$ 
Then, for each $n\in\mathbb{Z}_{\ge 0}$, define a measure $\widehat{\nu}_n$ on $\Sigma_{A'}$ by
$$ \widehat{\nu}_n(E) = \widehat{\nu}\left(E \cap \bigcup_{j=0}^n A_j\right)$$
for $E \subset \Sigma_{A'}$.
\end{definition}
\noindent Intuitively, each $A_j$ consists of elements in $\Sigma_{A'}$ that correspond to a path in $\mathcal{G}$ that starts at $\ast$, enters a maximal component on exactly its $j$th step and then never leaves this component.
\begin{lemma}
There exists $0<\theta<1$ such that $ \left\| \widehat{\nu}_n -  \widehat{\nu} \right\|_{TV} = O(\theta^n),$
as $n\to\infty$. The implied constant is independent of $n$.
\end{lemma}
\begin{proof}
We claim that
$$ \widehat{\nu}\left(\bigcup_{j> n} A_j\right) \to 0$$
exponentially quickly as $n\to\infty$. To see this, note that the number of length $n$ paths in $\mathcal{G}$ that start at $\ast$ and do not enter a maximal component is $O((\lambda -\delta)^n)$ for some $0<\delta < \lambda $. Combining this observation with Lemma $4.4$ implies that there exists $C>0$ independent of $j,n$ such that
$$ \widehat{\nu}\left(\bigcup_{j> n} A_j\right) \le C \sum_{j> n} \left(\frac{\lambda-\delta}{\lambda}\right)^j.$$
This proves the claim. Along with Lemma $4.4$, this shows that $Y \backslash \cup_{i=1}^m Y_i$ can be written as a countable union of zero $\widehat{\nu}$ measure sets. Hence $\widehat{\nu}\left(Y\backslash\cup_{i=1}^m Y_i\right) =0$ and for any $E \subset Y$,
$$ \widehat{\nu}(E) - \widehat{\nu}_n(E) = \widehat{\nu}\left(E\cap \bigcup_{j> n} A_j\right) \le \widehat{\nu} \left(\bigcup_{j> n} A_j\right).$$
Applying the claim a further time concludes the proof.
\end{proof}
We end this section by observing that, for any $E \subset \cup_i \Sigma_{B_i}$,
\begin{equation}
\sigma^j_\ast \widehat{\nu}(E) =\sigma^j_\ast\widehat{\nu}_j(E).
\end{equation}
We are now ready to prove our results.
\section{Proofs of Results}
Throughout the rest of the paper, suppose that $\varphi: G \to \mathbb{R}$ satisfies Condition $(1)$ and Condition $(2)$ and let $f: \Sigma_A \to \mathbb{R}$ be the function related to $\varphi$. Fix a bounded subset $H \subset C(G)$ (i.e. $\sup_{g\in H}\{|g|\} < \infty$).  \\ \indent

We begin by noting that Theorem $1.2$ is equivalent to the fact that there exists $\Lambda \in \mathbb{R}$ for which the set
$$\mathcal{U}_\Lambda = \left\{[\widetilde\gamma] \in \partial G: \lim_{n\to\infty} \frac{\varphi(\widetilde{\gamma}_n)}{n} = \Lambda\right\}, $$
is well-defined and has full $\nu$ measure.
\begin{lemma}
For any $\Lambda \in \mathbb{R}$ the set  $\mathcal{U}_\Lambda$ is well-defined and $G$-invariant.
\end{lemma}
\begin{proof}
Since $\varphi$ is Lipschitz in the right word metric, if $[\gamma] \in \partial G$ and $g \in G$, then there exists $C>0$ for which
$$|\varphi(\gamma_n) - \varphi(g\gamma_n)| \le C|g|$$
uniformly for $n\in\mathbb{Z}_{\ge0}$. Hence 
$$\lim_{n\to\infty} \frac{\varphi(\gamma_n)}{n} = \Lambda \hspace{2mm} \text{ if and only if } \hspace{2mm} \lim_{n\to\infty} \frac{\varphi(g\gamma_n)}{n}=\Lambda.$$
This proves $G$-invariance assuming that $\mathcal{U}_\Lambda$ is well-defined. To prove that $\mathcal{U}_\Lambda$ is well-defined we can follow the same argument as above, this time using that $\varphi$ is Lipschitz in the left word metric.
\end{proof}
We are now ready to prove Theorem $1.2$.
\begin{proof} [Proof of Theorem $1.2$]
Since the action of $G$ on $\partial G$ is ergodic with respect to $\nu$, it suffices, by Lemma $5.1$, to prove that there exists $\Lambda$ for which $\mathcal{U}_\Lambda$ has positive $\nu$ measure. Consider a maximal component $B_i$. By the ergodic theorem, $\mu(E_\Lambda) >0,$ where
$$E_\Lambda =\left\{y \in \Sigma_{B_i} : \frac{f^n(y)}{n}  \to \Lambda \text{ as } n\to \infty\right\}$$
and $\Lambda = \int_{\Sigma_{B_i}} f \ d\mu_i$. Hence by Proposition $4.6$ there exists $k \in \mathbb{Z}_{\ge 0}$ for which $\sigma_\ast^k\widehat{\nu}(E_\Lambda)>0$. We now note that if $y \in E_\Lambda$ and $x \in \bigcup_{n\ge0} \sigma^{-n}(\{y\}) $ then
$$\lim_{n\to\infty} \frac{f^n(x)}{n} \to \Lambda $$
as $n\to\infty$. Hence, 
$$\widehat{\nu}\left\{y \in Y: \frac{f^n(y)}{n}\to \Lambda \text{ as } n\to\infty\right\} \ge \sigma_\ast^k\widehat{\nu}(E_\Lambda) >0.$$
By Condition $(1)$, for $y \in Y,$ $f^n(y_n) = \varphi(h(y)_n) + O(1)$ where the implied constant is independent of both $n$ and $y$. Combining this with the fact that $h_\ast\widehat{\nu} =\nu$ implies that $\nu\left(\mathcal{U}_\Lambda\right) >0$ and thus concludes the proof.
\end{proof}
We now move on to the proof of Theorem $1.4$. By replacing $\varphi(\cdot)$ with $\varphi(\cdot) - \Lambda|\cdot|$ and $f(\cdot)$ with $f(\cdot) - \Lambda$, it suffices to prove Theorem $1.4$ under the assumption that $\Lambda =0$. We will assume this from now on. \\ \indent
The intuition behind our proof of Theorem $1.4$ is the following. By Proposition $4.6$, $\mu$ is obtained from averaging the pushforwards of $\widehat{\nu}$. If we could therefore, in some sense, reverse this averaging and express $\widehat{\nu}$ in terms of $\mu$, then we could use our knowledge of $\mu$ to learn about $\widehat{\nu}$. The relationship between these measures is particularly nice and allows us carry out such a procedure. \\ \indent Recall that we want to study the convergence of the following distributions.
\begin{definition}
Define, for $n\in\mathbb{Z}_{\ge0}$ and $x\in\mathbb{R}$,
$$R_n(x) = \nu \left\{ [\widetilde{\gamma}] \in \partial G : \text{ for all } \gamma \in [\widetilde{\gamma}] \text{ with } \gamma_0 \in H, \frac{\varphi(\gamma_n)}{\sqrt{n}} \le x \right\}$$
and 
$$N(x,\sigma) = \frac{1}{\sqrt{2 \pi } \sigma} \int_{-\infty}^x e^{-t^2/2\sigma} \ dt.$$
\end{definition}
\noindent We want to prove that there exists $\sigma^2 \ge 0$ for which
$$\|R_n(x) - N(x,\sigma)\|_\infty = O(n^{-1/4})$$
as $n\to\infty$. To simplify notation we will express this as $R_n = N(\sigma) + O(n^{-1/4})$.
We will use the following fact multiple times.
\begin{lemma}
Let $F_n,H_n: \mathbb{R} \to \mathbb{R}$ be sequences of distributions and suppose that $k_n, l_n$ are sequences of integers with $k_n \to \infty$ and $l_n \to \infty$ as $n \to \infty$. Suppose further that there exists a constant $C>0$ independent of $n$ and $x$ such that
$$H_n(x - Cl_n^{-1}) \le F_n(x) \le H_n(x+Cl_n^{-1}),$$
for all $n,x$. Then, if $H_n = N(\sigma) +O(k_n^{-1})$, we have that $F_n = N(\sigma) +O(k_n^{-1},l_n^{-1})$.
\end{lemma}
\begin{proof}
This is a simple consequence of the fact that the derivative of $N(\sigma)$ is uniformly bounded.
\end{proof}
Our aim is to construct a sequence of distributions on $Y$ with respect to $\widehat{\nu}$ from which we can gain an understanding of the $R_n.$ The following two lemmas are the first step in achieving this. The first lemma is an easy consequence of the hyperbolicity of $G$ and so we exclude the proof.
\begin{lemma}
There exists $C>0$ such that
$$\sup_{\substack{\gamma,\gamma' \in [\widetilde{\gamma}] \\ \gamma_0,\gamma'_0 \in H}} \sup_{n\in \mathbb{Z}_{\ge 0}} \{ d_L(\gamma_n, \gamma'_n)\} < C$$
uniformly for $[\widetilde{\gamma}] \in \partial G$.
\end{lemma}
Using this lemma we obtain.
\begin{lemma}
Define, for $n\in \mathbb{Z}_{\ge0}$ and $x\in \mathbb{R}$,
$$\widetilde{R}_n(x)= \nu \left\{ [\widetilde{\gamma}] \in \partial G : \text{ for some } \gamma \in [\widetilde{\gamma}] \text{ with } \gamma_0 \in H, \frac{\varphi(\gamma_n)}{\sqrt{n}} \le x \right\}.$$
Then, if $\widetilde{R}_n = N(\sigma) +O(n^{-1/4}),$ we have that $R_n = N(\sigma) +O(n^{-1/4})$.
\end{lemma}

\begin{proof}
Clearly $R_n(x) \le \widetilde{R}_n(x)$ for all $x \in \mathbb{R}$ and $n\in\mathbb{Z}_{\ge0}$. Also, by the previous lemma and the fact that $\varphi$ is Lipschitz in the $d_L$ metric, there exists $C>0$ independent of $x$ and $n$ such that
$$\widetilde{R}_n(x-Cn^{-1/2}) \le R_n(x),$$
for all $x,n$. Combining these two bounds and applying Lemma $5.3$ concludes the proof.
\end{proof}
The previous two lemmas show that, without loss of generality, we may assume that the identity element of $G$ belongs to $H$. We will assume this from 
now on. We can now construct distributions on $Y$ from which we can deduce the convergence of $R_n$. Recall that given $y \in  Y$, $h(y)_n$ for $n\in\mathbb{Z}_{\ge0}$ denotes the $n$th group element in the geodesic ray determined by $y$.
\begin{definition}
Define distributions
$$H_n(x) = \widehat{\nu}\left\{ y \in \bigcup_i Y_i : \frac{\varphi(h(y)_n)}{\sqrt{n}} \le x \right\}$$
for $n\in \mathbb{Z}_{\ge 0}$ and $x\in\mathbb{R}$.
\end{definition}
The following lemma shows that to prove Theorem $1.4$, it suffices to prove the analogous statement for the distributions $H_n$.
\begin{lemma}
If $H_n = N(\sigma) +O(n^{-1/4})$ then $R_n = N(\sigma) +O(n^{-1/4})$.
\end{lemma}
\begin{proof}
It is proven in \cite{calnotes} that $h$ is surjective, see Lemma $3.5.1$. Hence there exists $K>0$ independent of $n,x$ such that
\begin{align*}
H_n(x) &\le  \widehat{\nu} \left(h^{-1}\left\{ [\widetilde{\gamma}] \in \partial G : \text{ for some } \gamma \in [\widetilde{\gamma}] \text{ with } \gamma_0 \in H, \frac{\varphi(\gamma_n)}{\sqrt{n}} \le x \right\}\right)\\
&\le H_n(x+Kn^{-1/2}),
\end{align*}
for all $n \in \mathbb{Z}_{\ge 0}$ and $x \in \mathbb{R}$. Since $h_\ast\widehat{\nu} = \nu$,
$$  \widehat{\nu} \left( h^{-1}\left\{ [\widetilde{\gamma}] \in \partial G : \text{ for some } \gamma \in [\widetilde{\gamma}] \text{ with } \gamma_0 \in H, \frac{\varphi(\gamma_n)}{\sqrt{n}} \le x \right\}\right) = \widetilde{R}_n(x)$$ 
and applying Lemmas $5.3$ and $5.4$ completes the proof.
\end{proof}
The next step is to study the $H_n$. We do this by constructing distributions on $\cup_i \Sigma_{B_i}$ with respect to $\mu$ and then, by relating $\mu$ to $\widehat{\nu}$, use these to understand the $H_n$ distributions. To simplify notation, we define, for $x\in \mathbb{R}$ and $n\in\mathbb{Z}_{\ge0}$,
$$E_n(x) = \left\{ y \in \bigcup_i Y_i: \frac{f^n(y)}{\sqrt{n}} \le x \right\} \subset Y.$$
The following lemma along with Proposition $4.6$ will allow us to compare the $\widehat{\nu}$ and $\mu$ measures.
\begin{lemma}
For any sequence of integers $k_n$ such that $k_n\to\infty$ as $n\to\infty$,
$$ \frac{1}{k_n} \sum_{j=0}^{k_n} \widehat{\nu}_j(E_n(x)) = \widehat{\nu}(E_n(x)) +O(k_n^{-1}),$$
where the implied constant is independent of $n,x$.
\end{lemma}
\begin{proof}
By Lemma $4.8$ there exists $0<\theta<1$ such that for each $j \in \mathbb{Z}_{\ge 0}$, 
$$\widehat{\nu}_j(E_{n}(x)) = \widehat{\nu}(E_n(x)) + O(\theta^j),$$ where the implied constant is independent of $j$, $n$ and $x$. Taking the average of $\widehat{\nu}_1(E_{n}(x)),...,\widehat{\nu}_{k_n}(E_n(x))$ and letting $n\to\infty$ gives the result.
\end{proof}
We now, using work from \cite{me}, describe how $f$ distributes over $\Sigma_A$ with respect to the measure $\mu$. Along with the previous lemma, this will allow us to deduce the convergence of the $H_n$ distributions.
\begin{proposition}
There exists $\sigma^2 \ge 0$ such that for each $x \in \mathbb{R}$,
$$\mu\left\{y \in \bigcup_i \Sigma_{B_i}: \frac{f^n(y)}{\sqrt{n}} \le x\right\} =N(x,\sigma) + O(n^{-1/2})$$
as $n\to\infty$ and the above error term is uniform in $x \in \mathbb{R}$. Furthermore, $\sigma^2>0$ if and only if 
$$\left\{[\gamma] \in \partial G: \{\varphi(\gamma_n):n\in\mathbb{Z}_{\ge0}\} \text{ is unbounded}\right\}$$
is non-empty.
\end{proposition}
\begin{proof}
By Proposition $4.2$, the measure $\mu$ is a weighted sum of the measures of maximal entropy $\mu_i$ on each maximal component $B_i$. We obtain a central limit theorem, with mean $\Lambda_i$ and variance $\sigma_i$, for $\mu_i$ and $f$ on each $\Sigma_{B_i}$. Proposition $6.2$ from \cite{me} uses an argument of Calegari and Fujiwara to show that $\Lambda_i$ and $\sigma_i$ do not depend on the maximal component $B_i$ (and by assumption $\Lambda_i =0$ for each $i=1,...,m$). From this and the Berry-Esseen Theorem for subshifts of finite type \cite{cp} we obtain the desired central limit theorem, with error term, for $\mu$ and $f$. The criteria for positive variance follows from Lemma $7.2$ and Proposition $7.7$ of \cite{me}.
\end{proof}
We are now ready to prove Theorem $1.4$.

\begin{proof} [Proof of Theorem $1.4$]
By Lemma $5.7$ it suffices to prove that for $x \in \mathbb{R}$
$$ H_n(x) = N(x,\sigma)  +O(n^{-1/4})$$
as $n\to\infty$.\\ \indent
We begin by applying Proposition $4.6$ and Proposition $5.9$ to deduce that for any integer valued sequence $k_n$, with $k_n\to\infty$ as $n\to\infty$,
\begin{equation}
\frac{1}{k_n} \sum_{j=0}^{k_n} \sigma_\ast^j \widehat{\nu}\left\{y\in \bigcup_i \Sigma_{B_i} : \frac{f^n(y)}{\sqrt{n}} \le x \right\} =N(x,\sigma)  +O(k_n^{-1}, n^{-1/2}), 
\end{equation}
as $n\to\infty$, uniformly for $x \in\mathbb{R}$. We then define, for $n\in \mathbb{Z}_{\ge0}$ and $x \in \mathbb{R}$, 
$$ C_{n}^{\pm}(x) = \left\{y \in \bigcup_i \Sigma_{B_i}: \frac{f^n(y)}{\sqrt{n}} \le x \pm \frac{2k_n|f|_\infty}{\sqrt{n}} \right\}.$$
If we suppose further that $k_n = o(\sqrt{n})$, then expression $(5.1)$ implies that
\begin{equation}
\frac{1}{k_n}\sum_{j=0}^{k_n} \sigma^j_\ast\widehat{\nu}(C_{n}^{\pm}(x)) = N(x,\sigma)  + O(k_n n^{-1/2},k_n^{-1}).
\end{equation}
We now note that, by containment,
\begin{equation}
\sigma^j_\ast\widehat{\nu}_j(C_n^-(x)) \le \widehat{\nu}_j(E_{n}(x)) \le \sigma^j_\ast\widehat{\nu}_j(C_n^+(x))
\end{equation}
for all $n$, $j \le k_n$ and $x$. Recall that, by $(4.2)$, $\sigma^j_\ast\widehat{\nu}(C_n^{\pm}(x)) = \sigma^j_\ast\widehat{\nu}_j(C_n^{\pm}(x))$ for all $n,x$. 
Hence, if we choose $k_n = \lfloor n^{1/4} \rfloor$, then $(5.2)$ along with inequality $(5.3)$ imply that
$$\frac{1}{k_n}\sum_{j=0}^{k_n} \widehat{\nu}_j(E_{n}(x)) = N(x,\sigma)  + O(n^{-1/4})$$ 
and so by Lemma $5.8$, 
$$ \widehat{\nu} (E_n(x))= N(x,\sigma)  + O(n^{-1/4}).$$
Lastly, using Lemma $5.3$ and the fact that, for $y \in Y,$ $f^n(y_n) = \varphi(h(y)_n) + O(1)$, it is easy to see that
$$H_n(x) =  \widehat{\nu}(E_n(x)) + O(n^{-1/2}) =  N(x,\sigma)  + O(n^{-1/4}),$$
concluding the proof.
\end{proof}

\begin{remark}
The $O(n^{-1/4})$ error term arises due to the fact that $\nu$ is supported on $Y$ whereas $\mu$ is supported $\cup_i \Sigma_{B_i}$. To pass the central limit theorem in Proposition $5.9$ to one for $\nu$ and $Y$, we need to compare the values $f$ takes on $Y$ to the values $f$ takes on  $\cup_i \Sigma_{B_i}$. This comparison introduces an error term that can be seen explicitly as the $2k_n|f|_\infty n^{-1/2}$ terms in the sets $C_{n}^\pm(x)$. In the case that $A$ is aperiodic (or irreducible) this term is no longer needed since for any $y\in Y$, $\sigma(y)$ belongs to the only (necessarily maximal) component.
\end{remark}
In \cite{bow}, Bowen and Series provide a geometrical condition for Fuchsian groups and their generating sets that guarantees the existence of a coding $\Sigma_A$ described by an aperiodic matrix. This condition is satisfied by the fundamental groups of compact hyperbolic surfaces (i.e. surface groups) with presentation 
$$\langle a_1,\ldots,a_g,b_1,\ldots,b_g | \prod_{j=1}^g [a_j,b_j] \rangle$$
where $g \ge 2$ is the genus of the surface. Free groups equipped with their canonical generating set also satisfy this condition. The above remark then implies the following.
\begin{corollary}
If $G$ and $\varphi:G \to \mathbb{R}$ satisfy the hypotheses of Theorem $1.4$ and $G$ is a free group or surface group equipped with the generating set described above, then the error term in Theorem $1.4$ can be improved to $O(n^{-1/2})$. 
\end{corollary}

\begin{remark}
It seems plausible that the optimal error term in Theorem $1.4$ is $O(n^{-1/2})$. The author has not pursued this however.
\end{remark}

\end{document}